\documentclass[graybox]{svmult}

\usepackage{mathptmx}       
\usepackage{helvet, amsmath, amsfonts}         
\usepackage{courier}        
\usepackage{type1cm}        
\usepackage{makeidx}         
\usepackage{graphicx}        
\usepackage{multicol}        
\usepackage[bottom]{footmisc}

\makeindex             
\def\R{\mathbb R}
\def\l{\lambda}
\def\a{\alpha}

\def\b{\beta}

\def\build#1_#2^#3{\mathrel{\mathop{\kern 0pt#1}\limits_{#2}^{#3}}}
\def\ds{\displaystyle}
{\catcode"40=11
\gdef\Arc#1{\setbox\z@ \hbox{$#1$}\mathop{#1\mkern3mu }^{\@rc{\z@}}\limits}
\gdef\@rc#1{\setbox\@ne\hbox{$\braceld$} \hbox to\wd#1{$\braceld\leaders\hrule height\ht\@ne\hfill\bracerd$}} }

\newcounter{contador}
\newtheorem{propo}[contador]{Proposition}
\newtheorem{teo}[contador]{Theorem}
\newtheorem{lem}[contador]{Lemma}
\newtheorem{nota}[contador]{Remark}

\newtheorem{corol}[contador]{Corollary}

\begin{document}

\title*{A QRT-system of two order one homographic difference equations: conjugation
 to rotations, periods of periodic solutions, sensitiveness to initial conditions}
 \titlerunning{A QRT-system of two order one homographic difference equations}  \author{Guy Bastien and Marc Rogalski}

\institute{Guy Bastien \at Institut Math\'ematique de Jussieu-Paris Rive Gauche, University Pierre et Marie Curie and CNRS, \email{guy.bastien@imj-prg.fr}
\and Marc Rogalski \at Laboratoire Paul Painlev\'e, University of Lille 1 and CNRS, and IMJ-PRG \email{marc.rogalski@upmc.fr}}

\maketitle

\abstract{We study the ``homographic" system of order one difference equations in $\R_*^{+^2}$ 
$$u_{n+1}u_n=c+\frac{d}{ v_n},\,\,\,\,\,\,v_{n+1}v_n=c+\frac{d}{u_{n+1}},$$ for $c, d>0$. We prove that the orbit $(M_n)_n=\big((u_n,v_n)\big)_n$ of a point
$M_0=(u_0,v_0)$ is contained in an invariant cubic curve, and that the restriction to the positive part of this cubic of the associated dynamical system is
conjugated to a rotation on the circle. 
For a dense invariant set of initial points the solutions are periodic, and if $c=1$ (this is always possible) every
integer $n\geq N(d)$ is the minimal period of some periodic solution. Moreover, every $n\geq11$ is the minimal period of some solution for some $d>0$, and we 
find exactly the set of such minimal periods between 2 and 10. We study the
associated dynamical system, and prove that there is a chaotic behavior on every compact set of $\R_*^{+^2}$
 not 
 containing the equilibrium.}
 \vskip 1mm
  Keywords: difference equations, periodic solutions, sensitiveness to initial conditions

\section{A geometric definition for an homographic system of difference equations}
\label{sec:1}
First we remark that in the system of the abstract we can suppose $c=1$ (put $u_n=u'_n\sqrt c$ and $v_n=v'_n\sqrt c$). From now on we take $c=1$.

\subsection{From a family of cubic curves to a system of difference equations}
\label{subsec:1.1}

Let be the family of cubic curves ${\cal C}_K$ in the plane, equations of which are
\begin{equation}
xy(x+y)+(x+y)+ d-Kxy=0,\hskip 4mm\textrm {with}\hskip 4mm d>0,\,\,K\in\R. 
\end{equation}
We define a map $F:\R_*^{+^2}\to\R_*^{+^2}$ by the
following geometric construction:
\noindent if $M=(x,y)\in\R_*^{+^2}$, we consider the curve ${\cal C}_K$ which contains $M$; the horizontal line passing through $M$ cuts ${\cal C}_K$
in a second point $M'$; now the vertical line passing through $M'$ cuts again ${\cal C}_K$, and this intersection is
$F(M)$ (remark that the infinite points in horizontal and vertical directions are on the curve). It is easy to see that $F(x,y):=(X,Y)$ is defined by
\begin{align}
\left\{
\begin{aligned}
Xx&=1 +\frac d y ,\\  Yy&=1 +\frac d X 
\end{aligned}
\right.
\end{align}
or
 \begin{equation}
(X,Y)=\Big(\frac{y+ d}{xy},\frac{ d xy+y+d}{ y( y+ d)}\Big). 
\end{equation}
The map $F$ is defined on $\R_*^{+^2}$, with values in $\R_*^{+^2}$, and
it is easy to see that $F$ is an homeomorphism of $\R_*^{+^2}$ onto itself, satisfying
\begin{equation}
F^{-1}=S\circ F\circ S,  
\end{equation}
where $S$ is the symmetry with respect to the diagonal.
By definition the cubic curves ${\cal C}_K$ are {\sl invariant} under the action of $F$, and the quantity
\begin{equation}
G(x,y):=x+y+\frac1x+\frac1y+\frac{ d}{ xy}
\end{equation}
 is {\sl invariant} under the action of $F$: the curve ${\cal C}_K$ is the $K$-level set of $G$.
\vskip 1mm
If $M_0:=(u_0,v_0)\in\R_*^{+^2}$, then its iterated points $M_n:=(u_n,v_n)=F^n(M_0)$ are the solutions of the system of two order one difference
equations in $\R_*^{+^2}$
\begin{align}
\left\{
\begin{aligned}
u_{n+1}\,u_n&=1+\frac{ d}{ v_n},\\  v_{n+1}\,v_n&=1+\frac{ d}{u_{n+1}}, 
\end{aligned}
\right.
\end{align}
 or
 \begin{align}
 \left\{
 \begin{aligned}
u_{n+1}&=\frac{v_n+ d}{ u_nv_n}\\ v_{n+1}&=\frac{ d u_nv_n+v_n+d}{ v_n( v_n+ d)}. 
\end{aligned}
\right.
\end{align}
Thus the orbit of $M_0$ is included
into the cubic ${\cal C}_K$ passing through $M_0$, and the function $G$ is an invariant for the system (6): the quantity $\ds u_n+v_n+\frac{1}{u_n}+\frac{1}{
v_n}+\frac{ d}{ u_nv_n}$ is independant of the integer $n$.
\vskip 1mm
In fact, the map $F$ is a particular case of the so called QRT-maps, introduced in [8] and particularly  studied in [6]. But our goal is to study the
behavior of the solutions of system (6), and in particular to find the possible periods of periodic points, and the chaotic behavior of the map $F$. For
this, we prefer to use methods analogous to these used in [2] or [12] instead of to use the general theory of QRT-maps. 

We start with the search of the fixed points of $F$, and of the critical points of the function $G$.

\subsection{Existence, unicity and equality of the critical point of $G$ and the fixed point of $F$}
\label{subsec:1.2}

The equations of the critical points of $G$ in $\mathbb R_*^{+^2}$ are

\[x^2y-y- d=0,\hskip 4mm y^2x-x- d=0.\]

 By difference one has $(x-y)(xy+ d)=0$, and so $x=y:=t$ satisfies the equation
$t^3-t- d=0$ which has exactly one solution $\ell>0$. We denote $L:=(\ell,\ell)$ this critical point of $G$.
\vskip 1mm 
A fixed point $(x,y)$ of $F$ satisfies the same equations as these of the critical point of $G$, so it has the form $(s,s)$ where $s$ satisfies
 the same equation
$s^3- s- d=0$. So we have already proved a part of the following result.
\vskip 1mm
\begin{lem}\label{ptfixcrit}
The map $F$ has exactly one fixed point $L=(\ell,\ell)$ where $\ds\ell\in\big]\max(1,\root 3\of{d}),1+\frac{ d}{2}\big[$ is the positive
solution of the equation
\begin{equation}
t^3-t- d=0.
\end{equation}
The function $G$ tends to $+\infty$ at the infinite point of $\mathbb R_*^{+^2}$, and has a unique critical point which is the
equilibrium $L$, where $G$ attains is strict minimum
\begin{equation} 
K_m=\frac{4\ell+3d}{\ell^2}=3\ell+\frac{1}{\ell}>4.
\end{equation}
\end{lem}
\vskip 1mm

\begin{proof}
The set in $\mathbb R_*^{+^2}$ defined by $\{G\leq M\}$ is compact, for $M>0$, because on it we have $x+y\leq M$ and $\ds xy\geq\frac{ d}{ M}$. This
proves that $G$ tends to $+\infty$ at the infinite point of $\mathbb R_*^{+^2}$.

If $P(t)=t^3-t- d$, then $P'=3t^2-1$ vanishes at $\ds\sqrt\frac{1}{3}$, and one has $P(1)=- d$ and $P(\root 3\of{d})=-\root 3\of{d}$. Since the
curve $y=P(t)$ is, on $\ds\big[\sqrt\frac{1}{3},+\infty\big[$, above its tangent at the point $(1,- d)$, it is easy to see that
$\ds\ell\in\big]\max(1,\root 3\of{d}),1+\frac{ d}{2}\big[$. Formulas (9) are obvious, and the other points are previously proved.\hskip 4mm \qed
\end{proof}

\vskip 1mm 
Now Lemma \ref{ptfixcrit}  has an important consequence, which is a direct application of a result of [3] generalized in [5].
\vskip 1mm
\begin{propo}\label{PermanentStable}
 The solutions of system (6) are permanent; if $(u_0,v_0)\not=L$, then the solution diverges. The equilibrium $L$ is localy stable.
Moreover, for
$K>K_m$ the positive component
${\cal C}_K^+$ of the cubic  ${\cal C}_K$ is diffeomorphic to the circle $\mathbb T$ and surrounds the point $L$. 
\end{propo}

\section{The group law on the cubic ${\cal C}_K$ and the dynamical system (2)}
\label{sec:2}
\vskip 2mm 
We will interpret the restriction of the map $F$ to the positive part ${\cal C}_K^+$ of ${\cal C}_K$ with the chord-tangent law on the cubic.
\vskip 2mm
\subsection{Study of the cubic curve ${\cal C}_K$}
\label{subsec:2.1}
The following lemma gives the essential facts about the cubic curve ${\cal C}_K$.
\vskip 1mm
\begin{lem}\label{ToutSurCubic}
For $K>K_m$, we have the following properties:

\noindent{\bf (1)} the cubic ${\cal C}_K$ is non-singular;

\noindent{\bf (2)} the cubic ${\cal C}_K$ has three asymptotes: the two axis $x=0$ and $y=0$, and the line $x+y=K$, which is an inflection tangent at the
infinite  point 
$D:=(1,-1,0)$ (in projective coordinates);

\noindent{\bf (3)} the cubic ${\cal C}_K$ cuts the axes at the points $A:=(-d,0)$ and $B:=(0,-d)$;

\noindent{\bf (4)} the positive component ${\cal C}_K^+$ of the cubic ${\cal C}_K$ is located in the triangular domain $x>0,\,\,y>0,\,\,x+y<K$; the part
${\cal C}_K\setminus\big\{{\cal C}_K^+\cup \{A\}\cup \{B\}\big\}$ of the cubic is contained in five triangular domains:
$x<0,\,\,y<0,\,\,x+y>-d\,;\,\,\,\,x>0,\,\,x+y<-d\,;\,\,\,\,y>0,\,\,x+y<-d\,;\,\,\,\,x<0,\,\,x+y>K\,;\,\,\,\,y<0,\,\,x+y>K$. The part $\overline{{\cal
C}_K}\setminus{\cal C}_K^+$ is connected in $\mathbb P^2(\mathbb R)$ ($\overline{{\cal C}_K}$ is the extension of ${\cal C}_K$ in $\mathbb P^2(\mathbb R)$).
\end{lem}

\begin{proof}
Points (2) and (3) are easy. Point (4) becomes from Lemma \ref{ptfixcrit} and from an other form for the equation of ${\cal C}_K$:
$xy(x+y-K)=-(x+y+d)$, so one has only to compare the signs of $xy$, $x+y-K$ and $x+y+d$. See the form of the cubic curve in Figure 1, which is proved in
Lemma \ref{FormCubic} below, in Subsection \ref{subsec:4.3}. The set
$\overline{{\cal C}_K}\setminus{\cal C}_K^+$ is connected by its infinite points.
\vskip 1mm
\begin{figure}
\includegraphics[scale=.65]{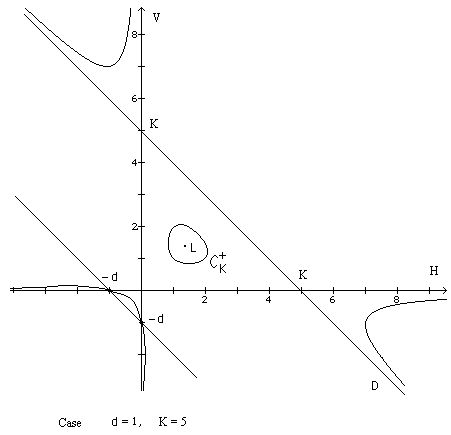}
\label{Figure 1}
\caption{}
\end{figure}

For point (1), the equations of a singular point $(x,y,t)$ are $f'_x=0,\,\,f'_y=0,\,\,f'_t=0$, where $f(x,y,t)=xy(x+y)+(x+y)t^2+dt^3-Kxyt$. We obtain
\begin{align}
\left\{
\begin{aligned}2xy+y^2+t^2-Kyt&=0,\\ 2yx+x^2+t^2-Kxt&=0,\\ 3dt^2+2(x+y)t-Kxy&=0.
\end{aligned}
 \right.
 \end{align}
 
 Obviously $t=0$ is not possible. The difference of the two
first equations gives the relation $(y-x)(x+y-Kt)=0$.

First suppose that $x\not=y$. Then we have two different (and symmetric) real singular points on the inflection asymptote, and so the curve splits in this
line and some symmetric hyperbola with equation $xy-\a t^2=0$. We write the equation of ${\cal C}_K$ under the form
$$(xy-\a t^2)(x+y-Kt)+(\a+1)(x+y)t^2+(d-\a K)t^3=0.$$ 
 If we take a finite point on the line $x+y-Kt=0$ we must have $(\a+1)K+d-\a K=0$, and thus $K=-d$,
which is impossible. 

Thus $x=y:=s$ satisfies the equations $Ks^2-4s-3d=0$ and $3s^2-Ks+1=0$. By elimination of $K$ between these two equations, we obtain $s^3-s-d=0$. So
$s=\ell$, and the first of the previous equations gives $\ds K=\frac{4\ell+3d}{\ell^2}=K_m$. In this case, $L$ is the singular point of ${\cal C}_{K_m}$: it
is a real isolated point of the curve, and ${\cal C}_{K_m}^+$ reduces to $\{L\}$.\hskip 4mm \qed
\end{proof}
\vskip 1mm

\subsection{The map $F$ and the group law on the cubic ${\cal C}_K$}
\label{subsection:2.2}

For $K>K_m$ there is on the cubic ${\cal C}_K$, more exactly on its extension $\widetilde{{\cal C}_K}$ in $\mathbb P^2(\mathbb C)$, a classical
chord-tangent group law (see [1] or [7]).  Denote
$H$ the infinite point in horizontal direction and
$V$ the infinite point in vertical direction. If we denote, for $P,\,Q\in\widetilde{{\cal C}_K}$, $P*Q$ the third point (finite or infinite) of
$\widetilde{{\cal C}_K}$ on the line $(PQ)$ (or on the tangent to $\widetilde{{\cal C}_K}$ at $P$ if $P=Q$), the chord-tangent group law, the zero element
of which is the point $V$, is
\begin{equation}
P\build{+}_{V}^{ }Q=(P*Q)*V.
\end{equation}
Note that in this case $V$ is not an inflection point of $\widetilde{{\cal C}_K}$; so the relation of alignment
of three points $P,\,Q,\,R\in\widetilde{{\cal C}_K}$ is 
$P\build{+}_{V}^{ }Q\build{+}_{V}^{ }R=V*V$.  Moreover the real part $\overline{{\cal C}_K}$ in $\mathbb P^2(\mathbb R)$ is a subgroup of the complex cubic.

Now from the geometric definition of the map $F$ given in Section \ref{subsec:1.1} we deduce the following result for the map $\overline{F}$, extension of $F$ to
 $\mathbb
P^2(\mathbb R)$ given by $\overline{F}(x,y,t)=\big(x(y+dt)^2,\,x(dxy+yt+dt^2),\,xy(y+dt)\big)$.
\vskip 1mm
\begin{propo}\label{FasAdd}

For $K>K_m$ the restriction of the map $\overline{F}$ to the cubic $\overline{{\cal C}_K}$ is nothing but the addition of the point
$H$ for the group law 
$\build{+}_{V}^{ }$: one has, for $M\in\overline{{\cal C}_K}$
\begin{equation}
\overline{F}(M)=M\build{+}_{V}^{ }H,\,\,\,\textit {and}\,\,\,M_n:=\overline{F}^{\,n}(M_0)=M_0\build{+}_{V}^{ }nH.
\end{equation}
So a solution of (6) with starting
point $M_0\in\overline{{\cal C}_K}$ is periodic with minimal period $n$ iff the infinite point $H$ is exactly of order $n$ in the group law
$\build{+}_{V}^{ }$ on $\overline{{\cal C}_K}$. If a point $M_0\in\overline{{\cal C}_K}$ is $n$-periodic, it is also the case for all the points of
$\overline{{\cal C}_K}$. The set ${\cal C}_K^+$ is stable under the action of $\overline{F}$, which coincides with $F$ on ${\cal C}_K^+$.
\end{propo}
\vskip 1mm
\begin{proof}
Relations (12) are obvious from the geometric definition of $F$ in Section \ref{subsec:1.1}. Then $M\in\overline{\cal C}_K$ has minimal period $n$ iff 
\begin{equation}
nH=V\,\,\,{\rm and}\,\,\, kH\not=V\,\,\,{\rm for}\,\,\,1\leq k\leq n-1.
\end{equation}
 But this condition depends only on $\overline{{\cal C}_K}$, that is on $K$,
and not on the particular point $M\in\overline{{\cal C}_K}$: this proves the last assertion of the proposition. \hskip 4mm\qed
\end{proof}
\vskip 1mm
\begin{nota}\label{Remark1}
In [3], exactly the same cubic curve ${\cal C}_K$ was used, as an invariant level set for the order 2 difference equation 
$\ds x_{n+2}x_n=1+\frac{a}{x_{n+1}}$, defined by the chord-tangent law $\build{+}_{D}^{ }$ on $\overline{{\cal C}_K}$, the zero element of which is the infinite
point $D$ on the asymptote $x+y=K$: one has in this case
$\,\,(x_{n+2},x_{n+1})=(x_{n+1},x_n)\build{+}_{D}^{ }V$.
\end{nota}
\vskip 1mm
\subsection{The group law on the cubic $\overline{{\cal C}_K}$ and periodic solutions of the system (6)}
\label{subsec:2.3}
We will see elementary that no solution of (6) has minimal period 2, 3 nor 4 .
\vskip 1mm
\begin{lem}\label{PasPeriodes234}
The only solutions of (6) which are 2, 3 or 4 periodic are constant (identical to the point $L$).
\end{lem}
\vskip 1mm
\begin{proof}
If a solution is 2-periodic, we have 
\[1+\frac{d}{ v_{n+1}}=u_{n+2}u_{n+1}=u_nu_{n+1}=1+\frac{d}{ v_n},\] 
and so $v_n$ is constant, equal to $v_0$; then $\ds 1+\frac{d}{u_{n+1}}=v_0^2$, and thus
$u_n$ is constant.

If a solution is 3-periodic, then $u_{n+2}u_{n+1}u_n=a$, $a$ constant, and $v_{n+2}v_{n+1}v_n=b$, $b$ constant. So we have $\,\,\ds u_{n+2}\Big(1+\frac{d}
{v_n}\Big)=a$ and $\,\,\ds v_{n+2}\Big(1+\frac{d}{ u_{n+1}}\Big)=b$, and thus $\,\,\ds\frac{d}{ v_n}=\frac{a}{ u_{n+2}}-1$ and $\,\,\ds1+\frac{d}{u_{n+1}}=\frac{b}
{v_{n+2}}= \frac{b}{ d}\Big(\frac{a}{ u_{n+4}}-1\Big)=\frac{b}{d}\Big(\frac{a}{ u_{n+1}}-1\Big)$; thus $\,\,\ds u_{n+1}=\frac{ab-d^2}{ b+d}$ is constant, and so $v_n$
is also constant. 

\vskip 1mm 
Now we will search if 4 may be a period of a solution of (6) by studying geometrically the equation $4H=V$.

First it is easy to see the opposite of a point $X$ of $\overline{{\cal C}_K}$ for the group law $\build{+}_{V}^{ }$:
\begin{equation}
-X=X*B,\,\,\,{\rm where}\,\,\,B=V*V=(0,-d,1).
\end{equation}
We denote $A=H*H=(-d,0,1)$, and ${\cal S},\,{\cal S}^+$ and ${\cal S}^-$ the three connected real
affine components of ${\cal C}_K\setminus{\cal C}_K^+$ located in the three domains $\{x+y<0\},\,\{x+y>K\}\cap\{x<0\}$ and $\{x+y>K\}\cap\{y<0\}$ (see Lemma \ref{ToutSurCubic}
 and Lemma \ref{FormCubic}).

 We have easily $2H\in{\cal S}^+$. We see that $-2H=(2H)*B$ (from (14)) is on ${\cal S}$. So we have $2H\not=-2H$, that is $4H\not=V$: there is no
4-periodic solution of (6) except $\{L\}$. \hskip 4mm\qed
\end{proof}
\vskip 1mm
In the following, we will transform the cubic curve $\overline{{\cal C}_K}$ in a standard cubic with equation $y^2=4x^3-g_2x-g_3$ and deduce of
this that the restriction of the map $F$ to ${\cal C}_K^+$ is conjugated to a rotation on the circle (remark that we know already that ${\cal C}^+_K$ is
diffeomorphic to the circle, from Proposition \ref{PermanentStable}). This result will give an other approach for the question of periodic solutions of (6). For algebraic
consistency, we work with the version of the cubic in homogeneous complex coordinates, that is in $\mathbb P^2(\mathbb C)$, and we denote as above
$\widetilde{{\cal C}_K}$ this extension of the cubic, with equation
$xy(x+y)+(x+y)t^2+ dt^3-Kxyt=0$.
\vskip 1mm
\section{Conjugation of $F_{|{\cal C}_K^+}$ to a rotation on the circle via Weierstrass' function $\wp$}
\label{sec:3}

 We start with the projective transformation ${\cal T}_1$ defined by
 \begin{align}
\left\{ 
 \begin{aligned}
 2X&=x+y,\\ 2Y&=y-x,\\ T&=x+y-Kt, 
 \end{aligned}
 \right.
 \end{align}
or \hskip 30mm $\ds x=X-Y,\,\, y=X+Y,\,\,t=\frac{2X-T}{ K}$,

\noindent  in order to transform the diagonal asymptote into the line at infinity and to
use the symmetry of the curve. The new cubic has for equation
\[Y^2T=8\frac{K+d}{ K^3}X^3+\frac{K^3-8K-12d}{ K^3}X^2T+2\frac{K+3d}{ K^3}XT^2-\frac{d}{K^3}T^3.\] 
 Now we make an affine transformation ${\cal T}_2$,
where $x,y,t$ are the new coordinates:
\begin{align}
 \left\{
 \begin{aligned}
 X&=\l x,\\ Y&=\l^2y,\hskip 8mm{\rm where}\hskip 8mm \l=\frac{1}{ K^{3/2}},\,\,\,\mu=2\frac{K+d}{ K^{3/2}}.\\\ T&=\mu t,
 \end{aligned}
 \right.
 \end{align}
 We obtain a new cubic with equation
\[y^2t=4x^3+(K^3-8K-12d)x^2t+4(K+3d)(K+d)xt^2-4d(K+d)^2t^3.\]
  We put
  \begin{equation}
  A:=K^3-8K-12d,
  \end{equation}
   and make an horizontal translation ${\cal T}_3$ defined by (with new variables $X,Y,T$)
   \begin{equation}
   X=x+\frac{A}{12}t,\,\,\,\,Y=y,\,\,\,\,T=t.
 \end{equation}  
    We obtain a new cubic $\Gamma_K$ with equation
  \begin{equation}  
    Y^2T=4X^3-g_2XT^2-g_3T^3,
    \end{equation}
    where 
   \begin{equation} 
    g_2=\frac{1}{12}(K^6-16K^4-24dK^3+16K^2);
    \end{equation}
     the value of $g_3$ will be unuseful.

\vskip 1mm

\begin{figure}
\includegraphics[scale=.65]{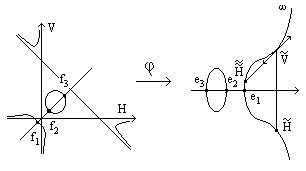}
\label{Figure 2}
\caption{}
\end{figure}

We interpret these three changes of variables as transformations between cubics, and so $\phi:={\cal T}_3\circ{\cal T}_2\circ{\cal T}_1$ is a projective real
transformation of $\overline{{\cal C}_K}$ onto $\Gamma_K$. So it transforms the two projective real connected parts of $\overline{{\cal C}_K}$ in two
projective real connected parts of
$\Gamma_K$, and so the positive compact part ${\cal C}_K^+$ of ${\cal C}_K$ onto the compact component $\Gamma_K^+$ of $\Gamma_K$. The three points of
${\cal C}_K$ on the diagonal, with coordinates $f_1,\,f_2,\,f_3$, become the three points of $\Gamma_K$ on the $X$-axis, with coordinates  $e_1,\,e_2,\,e_3$
(see Figure 2; we will see in lemma \ref{Formulaei} that the ``order" of the three points are inverted).

By the map $\phi$, the addition of $H$ on $\overline{{\cal C}_K}$ for the chord-tangent law $\build{+}_{V}^{ }$ with zero element $V$ is conjugated to the
addition of $\widetilde{H}$ on $\Gamma_K$ for the chord-tangent law $\build{+}_{\widetilde{V}}^{ }$ with zero element $\widetilde{V}$, where
$\widetilde{H}=\phi(H)$ and
$\widetilde{V}=\phi(V)$ (this is a general fact, but easy in our case because $\phi$ is a linear map in homogeneous coordinates).

It is easy to see that we have 
\begin{equation}
\widetilde{V}=
\begin{pmatrix}
K+d+\frac{A}{12}\\ (K+d)K^{3/2}\\ 1
\end{pmatrix}
\hskip 12mm{\rm and}\hskip 12mm 
\widetilde{H}=
\begin{pmatrix}
K+d+\frac{A}{12}\\ -(K+d)K^{3/2}\\1
\end{pmatrix}
\end{equation}

If $\omega$ is the infinite point on $\Gamma_K$ at the vertical direction, we know (see [1], [7]) that the standard group chord-tangent law
$\build{+}_{\omega}^{ }$ on
$\Gamma_K$ with $\omega$ as zero element is isomorphic to the standard group law on $\mathbb T^2$ (if we take real and complex points on $\Gamma_K$), via the
parametrization of $\Gamma_K$ by the Weierstrass' function $\wp$. But the addition of $\widetilde{H}$ on $\Gamma_K$ is not for its standard law, but for
$\build{+}_{\widetilde{V}}^{ }$. So we will make a supplementary isomorphism on $\Gamma_K$ in order to pass from a law to the other.

Recall that generally the chord-tangent law on a cubic, with zero element $Z$, is defined by
\[M\build{+}_{Z}^{ }P=(M*P)*Z,\]
 where $U*V$ denotes the third point of the intersection of the line $(UV)$ with the cubic.

Now we will define a group isomorphism $\psi$ of $(\Gamma_K,\widetilde{V})$ onto $(\Gamma_K,\omega)$ by
 \begin{equation}
 \psi:\Gamma_K\rightarrow\Gamma_K:M\mapsto M\build{+}_{\widetilde{V}}^{ }\omega=(M*\omega)*\widetilde{V}.
 \end{equation}

It is obvious that we have 
$\psi(\widetilde{V})=\omega$. The fact  that $\psi$ transforms the addition $\build{+}_{\widetilde{V}}^{ }$ in the addition
$\build{+}_{\omega}^{ }$ is not an obvious fact. It comes from a general fact about elliptic curves, because $\psi$ is birational (see [10]). But in the particular case of the 
addition of a point in the Weierstrass' cubic, there is an elementary computer assisted proof, which is given in Appendix (the existence of an elementary
proof is asserted in [10], page 21, without proof).

So the initial addition 
of $H$ on $(\overline{{\cal C}_K},V)$ is conjugated by $\psi\circ\phi$ to the addition of
$\widetilde{\widetilde{H}}=\psi(\widetilde{H})$ on $(\Gamma_K,\omega)$. Let be $X(K)$ the abscissa of the point $\widetilde{\widetilde{H}}$. It is known
(see [2], where one uses the parametrization of $\Gamma_K$ in the complex field by the Weierstrass' function) that the number of rotation of $F$ restricted
to ${\cal C}_K^+$ is the number
$\theta_d(K)$ in
$]0,1/2[$ given by the following integral formula which permits to invert  the Weierstrass' function $\wp$

\begin{equation}
2\theta_d(K)=\frac{\int_{_0}^{\sqrt\frac{e_1-e_3}{\nu}}\frac{\textrm {d}u}{\sqrt{(1+u^2)(1+\varepsilon u^2)}}
}{\int_0^{+\infty}\frac{\textrm {d}u}{\sqrt{(1+u^2)(1+\varepsilon u^2)}}},
\end{equation}
where one has 
\begin{equation}
\nu:=X(K)-e_1>0\hskip 4mm \textrm {and} \hskip 4mm\varepsilon:=\frac{e_1-e_2}{e_1-e_3}>0
\end{equation}
 (functions of $K$). So we have proved the following result
\vskip 1mm
\begin{teo}\label{ConjRot}
For $d>0$ and $K\in]K_m,+\infty[$ the restriction of the map $F$ to ${\cal C}_K^+$ is conjugated to the rotation on the circle $\mathbb T$ with angle $2\pi\theta_d(K)\in]0,\pi[$ given by formula (23). The map $K\mapsto\theta_d(K)$ is analytic on $]K_m,+\infty[$.
\end{teo}
\vskip 1mm
\begin{proof} 
The only thing to prove is the analyticity, and it results easily from the integral formula (23) because all the parameters in the integrals are analytic functions of $K$.\hskip 4mm \qed
\end{proof}
\vskip 1mm
\section{The possible periods of periodic solutions of system (6)}\label{sec:4}

We will study the number of rotation $\theta_d(K)$ given by formula (23) for $K>K_m$. Our goal is to find the limit of the function
 $K\mapsto\theta_d(K)$ when $K\to+\infty$ and when $K\to K_m$.
\vskip 1mm
\subsection{The limit of $\theta_d(K)$ at $+\infty$}\label{subsec:4.1}

It is first necessary to have asymptotic expressions of the numbers $e_1,\,e_2,\,e_3$ which appear in formulas (23) and (24), and which become from
$f_1,\,f_2,\,f_3$.
\vskip 1mm
\begin{lem}\label{Devfi}
 For $K>K_m$, the points $F_i$, $i=1,\,2,\,3$ of ${\cal C}_K$ on the diagonal exist, and their coordinates $f_i$  satisfy the inequalities
\[-\frac{d}{2}<f_1<0<f_2<\ell<f_3<\frac{K}{2}.\]
 When $K\to+\infty$, we have the asymptotic developments
 \begin{equation}
 f_1\sim-\sqrt{\frac{d}{K}},\,\,f_2\sim\sqrt{\frac{d}{K}},\,\,f_3=\frac{K}{2}-\frac{2}{K}+o(\frac{1}{ K}),\,\,f_2-f_1\sim 2\sqrt{\frac{d}{ K}},\,\,2f_3-K\sim-\frac{4}{K}.
\end{equation} 
 \end{lem}
 \vskip 1mm
 \begin{proof}
 The $f_i$'s are the solutions of the equation $2t^3-Kt^2+2t+d=0$, which has two positive roots because $K>\min G$, and a negative
root. The inequalities of the lemma are obvious on Figures 1 and 2. Then it is easy to deduce (25) from the three relations
\[f_1f_2f_3=-\frac{d}{2},\,\,f_1f_2+f_2f_3+f_3f_1=1,\,\,f_1+f_2+f_3=\frac{K}{2}.\hskip 4mm\qed\]
 \end{proof}

Now we calculate the $e_i$'s.
\vskip 1mm
\begin{lem}\label{Formulaei}
We have the formula
\begin{equation}
e_i=\frac{\mu}{\l}\frac{f_i}{2f_i-K}+\frac{A}{12},
\end{equation}
 and in particular $e_3<e_2<e_1$.
 \end{lem}
 \vskip 1mm
\begin{proof}
We take the images of the points $(f_i,f_i,1)$ by $\phi={\cal T}_3\circ{\cal T}_2\circ{\cal T}_1$, and obtain easily (26). Then the
decreasing monotony of the function $\ds x\mapsto\frac{x}{2x-K}$ on $\ds]-\infty,\frac{K}{2}[$ gives the final result.  \hskip 4mm \qed
\end{proof}
Now it is possible to obtain the asymptotic developments when $K\to+\infty$ of the parameters $\varepsilon$ and $\ds\sqrt{\frac{e_1-e_3}{\nu}}$ which
appear in the integrals of (23), and then to obtain the limit at $+\infty$ of $\theta_d(K)$.
\vskip 1mm
\begin{propo}\label{limitealinfini}
One has
\begin{equation}
\lim_{K\to+\infty}\theta_d(K)=\frac{3}{7}\,.
\end{equation}
\end{propo}
\begin{proof}
From formulas (16), (17) and (26), and relations (25), we have easily $\ds\varepsilon\sim\frac{16\sqrt {d}}{ K^{7/2}}$. We obtain also 
$\ds e_1-e_3\sim\frac{K^3}{4}$. Now it is necessary to get the value of $X(K)$. Let $\a$ be the first coordinate of $\widetilde{V}$, that is $\ds
\a=K+d+\frac{A}{12}$. The quantity $2\a+X(K)$ is the sum of the roots of the equation of degree 3 which express the abscissas of the intersection of
$\Gamma_K$ with its tangent at the point $\widetilde{V}$; let $Y=pX+q$ this tangent, the equation giving the abscissas of the intersections is
$(pX+q)^2=4X^3-g_2X-g_3$, that is $4X^3-p^2X^2+\cdots=0$. So we have $X(K)+2\a=p^2/4$. But the tangent at $\widetilde{V}$ to $\Gamma_K$ is the image by $\phi$ of
the tangent at $V$ to ${\cal C}_K$, the equation of which is $x=0$. After transformation we obtain for the final tangent the equation
$X-A/12=\l Y$; so  $p=1/\l=K^{3/2}$, and then $\ds X(K)+2\a=\frac{K^3}{4}$. So we have $\ds X(K)=K^3/4-2K-2d-\frac{1}{6}(K^3-8K-12d)=\frac{K^3}{12}-\frac{2}{3}K$.

Now easy calculations give $\ds\nu=X(K)-e_1=d-2\Big(1+\frac{d}{ K}\Big)\sqrt{\frac{d}{ K}}(1+o(1))\to d$ when $K\to+\infty$, and so
$\ds\frac{e_1-e_3}{\nu}\sim\frac{K^3}{4d}$.

If we take $\varepsilon\to 0$ as the variable, we have $\ds K\sim\frac{M_1}{\varepsilon^{2/7}}$, and
$\ds\sqrt{\frac{e_1-e_3}{\nu}}\sim\frac{M_2}{\varepsilon^{3/7}}$, with
$M_1$ and $M_2$ positive constants. At this point we can apply the following lemma of [2], and obtain the final result.  \hskip 4mm\qed
\end{proof}
\vskip 1mm
\begin{lem}\label{equivintegrale}
Let be $\l>0,\,\varepsilon>0,\,\beta>0$. For any map $\varepsilon\mapsto\psi(\varepsilon)=o(1)$ when $\varepsilon\to0$ and satisfying
$\l+\psi(\varepsilon)>0$, we put
\[N(\varepsilon,\l,\b)=\int_0^{\frac{\l+\psi(\varepsilon)}{\varepsilon^\b}}\frac{\textrm {d}u}{\sqrt{(1+u^2)(1+\varepsilon u^2)}}\,\,\,\,\textit {and}\,\,\,\,
D(\varepsilon)=\int_0^{+\infty}\frac{\textrm {d}u}{\sqrt{(1+u^2)(1+\varepsilon u^2)}}\,.\]
 Then, when $\varepsilon\to0$, we have
$D(\varepsilon)\sim(1/2)\ln(1/\varepsilon)$, and, if $\b<1/2$, $N(\varepsilon,\l,\b)\sim\b\ln(1/\varepsilon)$.
\end{lem}

\subsection{The limit of $\theta_d(K)$ at $K_m$}

It is known (see [4] or [12]) that we have the formula
\begin{equation}
\lim_{K\to K_m}\theta_d(K)=\frac{1}{2\pi}\arccos\Big(\frac{1}{2}\textrm {trace}\big(DF(L)\big)\Big),
\end{equation}
 where $DF$ is the jacobian matrix of $F$. With
formulas (2) and (8), it is easy to find the following result.
\vskip1mm
 \begin{propo}\label{limenKm}
 We have
 \begin{equation}
\theta_m(d):=\lim_{K\to
K_m}\theta_d(K)=\frac{1}{2\pi}\arccos\Big(\frac{1-2\ell^2-\ell^4}{2\ell^4}\Big)=\frac{1}{\pi}\arccos\Big(\frac{\ell^2-1}{2\ell^2}\Big),
 \end{equation}
  and the function
$d\mapsto\theta_m(d)$ is continuous on $]0,+\infty[$ and decreasing from $\ds\frac{1}{2}$ to $\ds\frac{1}{3}$ when $d$ increases from 0 to $+\infty$. We have 
  \begin{equation}
\textrm{ Im}(\theta_d)\supset\,\langle\frac{3}{7},\theta_m(d)\rangle:=\big]\min\Big(\frac{3}{7}, \theta_m(d)\Big),\max\Big(\frac{3}{7},
\theta_m(d)\Big)\big[.
 \end{equation} 
   One has the equivalence
   \begin{equation}
\theta_m(d)=[\textrm {resp.}<\textrm {or}>]\frac{3}{7}\iff d=[\textrm {resp.}>\textrm{ or}<]d_0:=\frac{2\sin(\pi/14)}{[1-2\sin(\pi/14)]^{3/2}}\approx1.076.
   \end{equation}
    When
$d=d_0$, we have $\ds\ell=\ell_0:=\frac{1}{\sqrt{1-2\sin\frac{\pi}{14}}}$.
\end{propo}
\vskip 1mm
If $d\not=d_0$, the function $K\mapsto\theta_d(K)$ is not constant, because its image contains the interval
$\ds\langle\frac{3}{7},\theta_m(d)\rangle$. But if
$d=d_0$, it would be possible that $\theta_d$ be constant, equal to $3/7$, and in this case all the solutions of (6) would be 7-periodic. But it is not the
case.
\vskip 1mm
\begin{propo}\label{7noglobalperiod}
The number 7 is not a common period to all the solutions of (6).
\end{propo}
\begin{proof}
If 7 would be a common period to every solution of (6), we would have $F^7=Id$, and so by (4), since $F^{-3}=S\circ F^3\circ S$, we
would have the relation $F^4\circ S=S\circ F^3$. If we take $u_0=v_0$, we would have $(u_4,v_4)=(v_3,u_3)$, and this relation gives $u_3v_3^2-v_3-d=0$. We
start with $(u_0,v_0)=(1,1)$, we calculate $(u_3,v_3)$, and the previous relation gives an equation for number $d$. With a computer, this equation has for
positive solution
$d\approx1.073$, which is different from $d_0$. This contradiction proves that when $d=d_0$ the solution with initial point $(1,1)$ is not 7-periodic. \hskip 4mm\qed
\end{proof}

\begin{corol}\label{notonetoone}
The map $K\mapsto\theta_{d_0}(K)$ is non constant and not one-to-one. There exists an open interval $I$ containing $d_0$ such that for
each $d\in I$ the map $\theta_d$ is not one-to-one and not constant.
\end{corol}
\begin{proof}
The first assertion results from Proposition \ref{7noglobalperiod}. For the second, suppose for example that $K\mapsto\theta_{d_0}(K)$ has a maximum
$M_0>3/7$ attained at $K_0>K_m\,$. Since $d\mapsto\theta_d(K_0)$ and $d\mapsto\theta_m(d)$ are obviously continuous, it exists $\eta>0$ such that for every
$d\in]d_0-\eta,d_0+\eta[$ we have $\theta_d(K_0)>(M_0+3/7)/2$ and
$\theta_m(d)<(M_0+3/7)/2$. Since $\ds\lim_{K\to+\infty}\theta_d(K)=3/7<(M_0+3/7)/2$, the function $\theta_d$ attains the value $(M_0+3/7)/2$ twice (at
least), and is not one-to-one. \hskip 4mm\qed 
\end{proof}
\vskip 1mm
\noindent \textbf {Problem.} Is the function $K\mapsto\theta_d(K)$ one-to-one if $\vert d-d_0\rvert$ is sufficiently large ?

\subsection{The possible periods of periodic solutions}\label{subsec:4.3}

We know from Theorem \ref{ConjRot} that the restriction of $F$ to ${\cal C}_K^+$ is conjugated to a rotation on the circle with angle $2\pi\theta_d(K)$. So,
if $\theta_d(K)$ is rational and equal to $\ds\frac{p}{ q}$ irreducible fraction, then the solutions with starting points on ${\cal C}_K^+$ are $q$-periodic,
with $q$ as minimal period, and the reciprocal is true. 

In the contrary, if  $\theta_d(K)$ is irrational, then every point of ${\cal C}_K^+$ has its orbit under the action of $F$ which is dense in ${\cal C}_K^+$,
and the converse is true.

How are distributed this two types of points, for a given $d$ ? What periods can appear, for a given $d$ ? The answers are given by the following result.

\begin{teo}\label{twodensesets}
Let $d$ be positive.

\noindent(1) It exists a partition of $\mathbb R_*^{+^2}\setminus\{L\}$ in two dense sets $A$ and $B$, each of them union of invariant curves ${\cal
C}_K^+$, such that every point in $A$ is periodic and every point in B has a dense orbit in the positive part of the cubic which passes through it.

\noindent(2) It exists an integer $N(d)$ such that every integer $q\geq N(d)$ is the minimal period of some solution of (6).
\end{teo}

\begin{proof}
\noindent\textbf{(1)} Put $\ds]a,b[:=\build{\textrm {Im}(\theta_d)}_{ }^{o}$. This interval is not empty ($a<b$) if $d\not=d_0$ because
$\theta_m(d)\not=3/7$, and if
$d=d_0$ because $K\mapsto\theta_{d_{\,0}}(K)$ is not constant. Since $K\mapsto\theta_d(K)$ is analytical, the function $\theta_d$ is constant on no
not-empty interval of
$]K_m,+\infty[$. So from the density of rational and irrational numbers in $]a,b[$ it results easily the density of the two sets
$\theta_d^{-1}(]a,b[\cap\mathbb Q)$ and $\theta_d^{-1}\big(]a,b[\cap(\mathbb R\setminus\mathbb Q)\big)$ in $]K_m,+\infty[$. From this, one see that the two sets
union of the curves ${\cal C}_K^+$ for $K$ in the two previous dense sets are dense.

\noindent\textbf{(2)} The set of minimal periods of periodic solutions is exactly the set of integers $q$ such that it exists a natural number
$p$ such that $\ds\frac{p}{q}$ is in $\textrm {Im}(\theta_d)$ and irreducible. So, if $\ds\frac{p}{q}$ lies in $]a,b[$ and is irreducible, $q$ is the minimal
period of some solution of (6). For finding such irreducible fractions, we fix the integer $q$ (the eventual period) and search for $p$ a \textit {prime number}
in the interval $]qa,qb[$, which is not a factor of $q$. It is known that the number of the distinct prime factors of a number $q$ is majorized by 
$\ds1.38402\,\frac{\ln q}{\ln(\ln q)}$ (see [9]). Denote $\pi(x)$ the cardinal of prime numbers not greater than $x$. So the number $P(q)$ of prime integers
between
$qa$ and
$qb$ which do not divide $q$ is at least
\[\pi(qb-1)-\pi(qa)-1.38402\,\frac{\ln q}{\ln(\ln q)}.\]
 From the prime number theorem (see [11]) we have 
$\ds c(q)\frac{q}{\ln q}\leq\pi(q)\leq C(q)\frac{q}{\ln q}$, with $c(q)\leq1\leq C(q)$  and $\ds\lim_{q\to+\infty}c(q)=\lim_{q\to+\infty}C(q)=1$. So we have 
\[\begin{aligned}P&\geq c(qb-1)\frac{qb-1}{\ln(qb-1)}-C(qa)\frac{qa}{\ln(qa)}-1.38402\frac{\ln q}{\ln(\ln q)}\\
&=\big(c(qb-1)b-C(qa)a\big)\frac{q}{\ln q}(1+\eta(q))
\end{aligned}\]
  where $\eta(q)\to0$ when $q\to+\infty$. So it exists a number $N(d)$ sufficiently large such that for every $q\geq N(d)$ one has $P(q)\geq1$, and so there exists a prime number $p$
such that
$\ds\frac{p}{ q}\in]a,b[$ and $p$ does not divide $q$. Thus $q$ is the minimal period of some solution of (6). \hskip 4mm \qed
  \end{proof}
  \vskip 1mm 
  Now we will search the set of integers which are minimal period of some solution of (6) for some value of $d>0$. The principle is the same, and
analogous to this of [2],  but the interval is now explicit :$]1/3,1/2[$, and so it is possible to improve the inequalities in the using of prime number theorem.
\vskip 1mm
\begin{teo}\label{LesPeriodes}
Every integer $q\geq11$ is the minimal period of some solution of system (6) for some $d>0$. Between 2 and 10, integers 2, 3, 4, 6, 10 are
minimal period of no non-constant solution of (6), for no $d$, the others: 5, 7, 8, 9 are minimal periods.
\end{teo}
\vskip 1mm
\begin{proof}

The proof is long, and we split it in four steps.
\vskip 1mm
\noindent\textbf {Step (1) of the proof.} First, we have, from Proposition \ref{limenKm}, the relation
\begin{equation}
\bigcup_{d>0}{\rm Im}(\theta_d)\supset]1/3,1/2[\setminus\{3/7\}.
\end{equation}
 Thus, for searching if an integer $q$ is a minimal period of some solution of
(6) for some $d>0$, it suffices to find some prime number $p\in]q/3,q/2[$ which is not a factor of $q$. There is an exception if $p/q=3/7$, but the only
possibility for this is $q=7$ and $p=3$. So in the following we suppose $q\not=7$.

We use an improvement of the prime number theorem, due to Rosser and Schoenfeld (see [11]):
\[\textrm {For}\,\,q\geq52,\,\,\,\frac{q}{\ln q}\leq\pi(q)\leq\Big(1+\frac{3}{2\ln q}\Big)\frac{q}{\ln q}.\]
So if the function
\begin{equation}
f(q):=\frac{(q/2)-1}{\ln((q/2)-1)}-\frac{q/3}{\ln(q/3)}\Big(1+\frac{3}{2\ln(q/3)}\Big)-1,38402\frac{\ln q}{\ln(\ln q)}-1
\end{equation}
 is positive for some $q\geq52$,
it exists $p\in]q/3,q/2[$ which is prime and does not divide $q$. Thus $q$ is a prime period of some solution of (6) for some $d>0$. Of course the
equivalent to $f(q)$ when $q\to+\infty$ is $\ds (1/6)\frac{q}{\ln q}$, so it is true that $f(q)>0$ for $q$ sufficiently large. But we wish to have a quantitative version of this.
 \vskip 1mm
We put

\[ f(x):=\frac{x/2-1}{\ln(x/2-1)}-\Big(1+\frac{3}{2\ln(x/3)}\Big)\frac{x/3}{\ln(x/3)}-1.38402\frac{\ln x}{\ln(\ln x)}-1.\]

We have $f(780)<0$ and $f(781)>0$. From the graph of $f$ on a computer, it seems that for every $x\geq781$ we have $f(x)>0$.
We give a mathematical proof of this fact.
\vskip 1mm
\noindent(a) We define the function $\ds x\mapsto g_k(x):=k\frac{x}{\ln x}-1.38402\frac{\ln x}{\ln(\ln x)}-1$, for $k>0$ to be chose. We study the monotonicity of $g_k$. 
We put $u:=\ln x$, and so have
\[g_k(x):=h_k(u)=k\frac{e^{u}}{u}-1.38402\frac{u}{\ln u}-1,\]
and choose $u\geq u_0$, that is $x\geq x_0:=e^{u_0}\geq52$ (and so $u_0\geq2$). We have
\[h_k'=ke^{u}\frac{u-1}{u^2}-1.38402\frac{\ln u-1}{\ln^2 u}\geq k(u_0-1)\frac{e^{u}}{u^2}-1.38402\frac{1}{\ln u}\]
and so we have

\[\frac{1}{k}\leq\frac{(u_0-1)e^{u}\ln u}{1.38402 u^2}:=\phi(u)\Longrightarrow h_k'(u)>0 .\]
\vskip 1mm
\noindent(b) Now we study the monotonicity of the function $\phi$. First, $\ds \big(\frac{e^{u}}{u^2}\big)'=\frac{e^{u}(u-2)}{u^3}\geq0$ if $u\geq2$. So $\phi$ is increasing
and we have
\[ \frac{1}{k}\leq\phi(u_0)\Longrightarrow \forall  u\geq u_0,\,\, h_k'(u)>0 ,\]
and so $g_k(x)>0$ for $x\geq x_0$ if $k$ satisfies the previous inequality and
$g_k(x_0)>0$.
\vskip 1mm
\noindent(c) Now we search a condition which will imply the inequality
\[\forall x\geq x_0,\,\,\frac{x/2-1}{\ln(x/2-1)}-\Big(1+\frac{3}{2\ln(x/3)}\Big)\frac{x/3}{\ln(x/3)}\geq k\frac{x}{\ln x}\Longleftrightarrow \forall x\geq x_0,\,\,f(x)\geq
g_k(x).\] 
Easy majorizations and minorations prove that a sufficient condition is to have, with $x_0\geq5$,
\[M(x_0):=\frac{1-2/x_0}{2}-\frac{1}{3}\Big(1+\frac{3}{2\ln(x_0/3)}\Big)\frac{1}{1-\frac{\ln3}{\ln x_0}}\geq k.\]
\vskip 1mm
\noindent (d) So the goal is to find an integer $x_0$ such that $\ds M(x_0)>\frac{1}{\phi(u_0)}$ for $u_0=\ln x_0$. We choose $x_0=2500$ and calculate the values
$M(2500)\approx0.0253\dots$ and
$\ds \frac{1}{\phi(\ln 2500)}\approx0.00241\dots$. So we can take $k=0.025$, and have to verify the sign of $g_k(x_0)$: $g_k(2500)=h_k(\ln 2500)\approx 2.7242>0$.

\vskip 1mm
\noindent (e) \textsl{In fine}, for $x\geq2500$ we have $f(x)>0$.
\vskip 1mm
\noindent (f) Now, using a computer, we see that $f(q)>0$ for every integer $q\in[781,\,2500]$.
\vskip 2mm
So our goal is to examine the integers in $[5, 780]$

First, remark the obvious inclusion:
\begin{equation}
\textrm {if}\,\, x\leq r\,\,\textrm {and}\,\,I_q:=]q/3,q/2[, \,\,\,\,]r/3,x/2[\,\,\subset\bigcap_{x\leq q\leq r}I_q.
\end{equation}
 We use this inclusion starting at $r=780$:
$780/3=260$; the smallest prime number majorizing 260 is $p=263$, so we take $x/2=264$, and so $x=528$. From (34), we have 
$p\in]780/3,528/2[\subset]q/3,q/2[$ for every $q\in[528,780]$, and $p$ does not divide such a $q$, which is so a minimal period. Now, we use the same method
starting from 527:
$527/3\approx175.66É$, and the smallest prime number majorizing this number is $p=179$, so we take $x/2=180,\,\,x=360$, and we can conclude that for all the
integers $q\in[360,527]$, the numbers $p/q$ are in $]1/3,1/2[$ and irreducible (because $p/q=p/pq'=1/q'\in]1/3,1/2[$ is not possible). So any integer
in [360, 527] is a minimal period. 

We continue this procedure, and finish with the fact that all the integers $q\geq24$ are minimal periods. The integers between 5 and 23 are examined one by
one, and only for 6 and 10 there is no irreducible fraction $p/6$ nor $p/10$ in $]1/3,1/2[$. Of course $3/7\in]1/3,1/2[$, but it is not {\sl a priori} in the set $\ds\bigcup_{d>0}\textrm {Im}(\theta_d)$. 
 
 \vskip 1mm
\noindent\textbf {Step (2) of the proof.}
We prove that 7 is a minimal period. We have to use the following result, which was previously already
used.
\vskip 1mm
\begin{lem}\label{FormCubic}
For $K>K_m$, the places and the shapes of the different branches of the cubic ${\cal C}_K$ are the same 
as in Figure 1.
\end{lem}
\vskip 1mm 
This Lemma will be proved after the end of the proof of Theorem \ref{LesPeriodes}.

We must also use the following result.
\vskip 1mm
\begin{lem}\label{relation-nH}
For every integer $n\geq 0$ we have the following formula for the group law on the curve $\overline{{\cal C}_K}$
\begin{equation}
-nH=S[(n+1)H],
\end{equation}
where $S$ is the symmetry with respect to the diagonal.
\end{lem}

\begin{proof}
We denote ($R_n$) this relation. For $n=0$, ($R_0$) is $V=S(H)$, which is true. Suppose that ($R_{n-1}$) is true: $-(n-1)H=S(nH)$. 
Since $F(M)=M\build{+}_{V}^{ }H$ and $F^{-1}(M)=M\build{-}_{V}^{ }H$, we have, using ($R_{n-1}$) and formula (4),
\[-nH=F^{-1}[-(n-1)H]=F^{-1}[S(nH)]=(S\circ F)(nH)=S[(n+1)H].\]
This recursion proves the lemma.  \hskip 4mm\qed
\end{proof}
\vskip 2mm
Now the condition for having 7 as a period is $7H=V$ or $-3H=4H$. But by Lemma \ref{relation-nH}, $-3H=S(4H)$, and so the condition is exactly $4H=S(4H)$, that is 
$4H=F_1$ or $4H=D$ (it is easy to see that each $nH\in\overline{{\cal C}_K}\setminus{\cal C}_K^+$). 
Let be $N_0$ the point of ${\cal S}^+$ with horizontal tangent (where $y$ is minimum on ${\cal S}^+$), which exists because ${\cal S}^+$ is a convex curve,
 by Lemma \ref{FormCubic} and Figure 1. If $2H\not=N_0$, one see easily that $4H=(2H*2H)*V$ is finite and on the branches ${\cal S}^+$ or ${\cal S}^-$, and so cannot be
 equal to $F_1$. So the unique possibility is $2H=N_0$, and then $4H=D$. Conversely, if $4H=D$, then $2H=N_0$, that is $N_0$ is in the vertical line through point $A$.

But it is easy to calculate the second coordinate of $2H$:
\begin{equation}
y_{_{2H}}=\frac{d^2+Kd+1}{ d}.
\end{equation}
So the condition for having 7 as a period is $f'_x(-d,y_{_{2H}})=0$, that is 
\begin{equation}
Kd(1-d^2)=d^4-d^2-1.
\end{equation}
So 7 is a period if and only if (37) is true for some $d>0$ and $K>K_m$. So we must have
\begin{equation} 
\ds1<d<\sqrt{\frac{1+\sqrt{5}}{2}}.
\end{equation}
And now we have $\ds K=\frac{d^4-d^2-1}{d(1-d^2)}$. So 7 is a period iff
\[\frac{d^4-d^2-1}{d(1-d^2)}>K_m=3\ell+\frac{1}{\ell}.\]
Since the map $d\mapsto\ell$ is bijective increasing from 1 to $+\infty$ when $d$ increases from 0 to $+\infty$, this condition can be translated as
\[\frac{(\ell^3-\ell)^4-(\ell^3-\ell)^2-1}{(\ell^3-\ell)[1-(\ell^3-\ell)^2]}>3\ell+\frac{1}{\ell}.\]
We put $\ell^2:=x$ and obtain the condition $u(x):=x^3-x^2-2x+1<0$. It is easy to see that the roots of $u$ satisfy the inequalities $x_2<0<x_1<1<x_m$. So the condition $u(x)<0$, for $x>1$, is $x<x_m=1.80193377É$ and then $\ell<\ell_m=1.34236É$, and so
 \[1<d<d_m=1,07649\cdots<\sqrt{\frac{1+\sqrt{5}}{2}}=1.2220\cdots\]
Then it is exactly for $\ds 1<d<d_m$ that it exists a (unique) $\ds K(d):=\frac{d^4-d^2-1}{d(1-d^2)}>K_m$ such that 7 is a 
period for the solutions of (6) which are on the curve ${\cal C}_{K(d)}$.
\vskip 1mm
\begin{nota}\label{dm=d0}
With some calculation it is possible to see that one has
\begin{equation}
d_m=d_0\hskip 4mm\textrm {and}\hskip 4mm \ell_m=\ell_0.
\end{equation}
\end{nota}

\vskip 2mm
In fact, look at the equation $w^7+1=0$, with $w\not=-1$. By setting $\ds X:=w+\frac{1}{ w}$, this equation is equivalent to the new
equation $X^3-X^2-2X+1=0$, whose roots are
\[2\cos\frac{5\pi}{7}=-2\sin\frac{3\pi}{14}<0<2\cos\frac{3\pi}{7}=2\sin\frac{\pi}{14}<1<2\cos\frac{\pi}{7}=2\sin\frac{5\pi}{14}.\]
So $\ds\ell_m^2=2\sin\frac{5\pi}{14}$, and we have easily the relation $\ds2\sin\frac{5\pi}{14}=\frac{1}{1-2\sin\frac{\pi}{14}}=\ell_0^2$ (see Proposition \ref{limenKm}), because
$\ds2\sin\frac{5\pi}{14}-2\sin\frac{3\pi}{14}+2\sin\frac{\pi}{14}=1$, sum of the roots of the equation $X^3-X^2-2X+1=0$.

\vskip 2mm
\begin{nota}\label{Thetad0inf3/7}
There is also a direct simple analytic proof that if $d$ is near $d_0$ the function $K\mapsto\theta_d(K)$ attains 
the value 3/7, by using the method of the proof of  Corollary \ref{notonetoone} (with the intermediate value theorem). The comparizon of this method with this previous geometric one and Remark \ref{dm=d0} gives the 
following qualitative result : 
\begin{equation}
\forall\, K>K_m,\,\,\,\theta_{d_0}(K)\leq3/7.
\end{equation}
\end{nota}

\vskip 2mm
\noindent\textbf {Step (3) of the proof.}
We solve the case $q=6$ by geometric method. First, we remark that $2n$ is a period of solutions on $\overline{{\cal C}_K}$ iff  $2nH=V$, or $nH=-nH$ in the group law, that is
$nH=(nH)*B$ (see formula (14)), and this is equivalent to the fact that $B$ belongs to the tangent to $\overline{{\cal C}_K}$ at the point $nH$. So, in the case of
$q=6=2\times3$, we will see that $B$ does not belong to the tangent at the point $3H$. 

First we localize $3H$. Since $H*H=A=(-d,0)$, we have $2H=A*V\in{\cal S}^+$ (see the proof of Lemma \ref{PasPeriodes234}). So $(2H)*H\in{\cal S}^+$, and
$3H=(2H)+H=[(2H)*H]*V$ is the vertical projection of $(2H)*H$ on ${\cal S}$, and so $3H$ is on the arc $\Arc{HAB}$ of ${\cal S}$. Now we use
the following lemma.

\begin{lem}\label{Inflections} 
There is exactly (for $K>K_m$) two real and finite inflection points $I,\,J$ on ${\cal C}_K$, symmetric with respect to the diagonal,
situated on ${\cal S}$, in the second and the fourth quadrant, $I$ on the arc $\Arc{HA}$ and $J$ on the arc $\Arc{BV}$. The
arc $\Arc{IABJ}$ of ${\cal S}$ is strictly convex, the arcs $\Arc{HI}$ and $\Arc{JV}$ also, and it is also the
case for ${\cal S}^+$  and ${\cal S}^-$.
\end{lem}
\vskip 1mm
With this lemma, we see that if $3H\in\Arc{IB}$, since $3H\not=B$, the tangent at $3H$ does not cut, by convexity, the arc $\Arc{IB}$, 
and so $B$ does not belong to this tangent. If now $3H\in\Arc{HI}$, by convexity the tangent at $3H$ has positive slope, and
cannot pass to $B$. So in each case we conclude that
$3H\not=-3H$, and so 6 is not a minimal period. 
\vskip 2mm
\noindent\textit{Proof of Lemma} \ref{Inflections}.   
The transformation $\phi$ of $\overline{{\cal C}_K}$ onto $\Gamma_K$ is real and projective, and so preserves the real inflection
points. So we search the inflection points of $\Gamma_K$. 

The positive part ($y\geq 0$) of this curve is defined by $y=\sqrt{P(x)}$, where $P(x)=4x^3-g_2x-g_3=4(x-e_1)(x-e_2)(x-e_3)$. One shows easily that the roots of
$(\sqrt{P})''$ are those of $2PP''-P'^2$. Let $h$ be this function. It is a degree 4 polynomial, and $h(e_i)=-P'^2(e_i)<0$. But $h\,'=2PP'''=48P$, which is
zero at the $e_i$'s. So the function $h$ is negative on $[e_3,e_1]$ and has only two real roots. The smallest gives no real inflection point, but the
greatest gives two inflection points on the unbounded branch of $\Gamma_K$, symmetric with respect to the $x$-axis. When we transform these two points by
$\phi^{-1}$, we obtain exactly two real finite inflection points $I$ and $J$, symmetric with respect to the diagonal. But, with Lemma \ref{FormCubic}, there are necessary
at least one inflection point on the arc
$\Arc{HA}$ and on the arc $\Arc{BV}$. So these points are exactly $I$ and $J$. The unicity of this inflection points and Lemma \ref{FormCubic} give the results on concavity and convexity asserted in Lemma \ref{Inflections}. \hskip 4mm\qed
\vskip 2mm
\noindent\textbf {Step (4) of the proof.}
Now we study if 10 is a minimal period (it is a period because 5 is period). So we suppose that $5H\not=V$ and $10H=V$, that is 
$5H=-5H=(5H)*B$, or : $B$ belongs to the tangent to $\overline{{\cal C}_K}$ at $5H$. We examine two cases.
\vskip 2mm
\noindent (4a) $3H=A$, that is $(2H)*H$ is on the vertical line through $A$, or $2H=N_0$ (see step (2) of the proof).
In this case we have $5H=\big((3H)*(2H)\big)*V=V*V=B$ and then $B*B=B*5H=5H=B$. Hence $B$ is an inflection point, and this is impossible by Lemmas \ref{FormCubic} and \ref{Inflections}.

\vskip 2mm
\noindent(4b) $3H\not=A$, that is $2H\not=N_0$. We look at two subcases.

(i) The point $2H$ is at the left side of $N_0$. In this case $(3H)*(2H)\in\Arc{BV}$, and then $5H\in\Arc{BV}$. But by hypothesis $5H\not=V$ and $5H$ cannot 
be equal to $B$ by (4a). So $5H$ belongs to the \textit {open} arc $\Arc{BJV}$. If $5H\in\Arc{JV}$, the tangent at $5H$ has a positive slope and cannot pass 
through $B$, and if $5H\in\Arc{BJ}$, the strict convexity of this arc proves that the tangent at $5H$ cannot pass through $B$.

(ii) The point $2H$ is at the right side of $N_0$. In this case $5H\in{\cal S}^+$, and no tangent at this branch passes through $B$, except if $5H=V$, 
which is excluded by hypothesis.

\vskip 1mm
\noindent So in every case of (4b) the tangent at $5H$ cannot pass through $B$, and then 10 is not a minimal period. 

This achieves the proof of Theorem \ref{LesPeriodes}. \hskip 4mm \qed
\end{proof}

\vskip 2mm
\noindent\textit{\textbf{Proof of Lemma}} \ref{FormCubic}.   
We will use the form of the curve $\Gamma_K$, or more exactly of the curve $\Gamma_K^*$ whose horizontal translated
 (by $\ds\frac{A}{12}$, see formulas (14), (16), (19)) is $\Gamma_K$, that is $\Gamma_K^*={\cal T}_2\circ{\cal T}_1(\overline{{\cal C}_K})$. We denote 
$V^*={\cal T}_2\circ{\cal T}_1(V)$, and with the same notation the image in $\Gamma_K^*$ of other remarkable points of $\overline{{\cal C}_K}$. With previous 
(see formula (22)) and new easy calculations, we find :
\begin{align}
\left\{
\begin{aligned}
A^*&=(d,-dK^{3/2},1),\\ B^*&=(d,dK^{3/2},1),\\ E_1^*&=(f_1^*,0,1),\\ H^*&=(K+d,-(K+d)K^{3/2},1),\\ V^*&=
(K+d,(K+d)K^{3/2},1),\\ D^*&=(0,1,0),
\end{aligned}
\right.
\end{align} 
with $0<f_1^*<d$ (use the inequalities of Lemma \ref{Devfi} and formula (26)). So these points are easy to place on the unbounded branch of the curve $\Gamma_K^*$, because $K>0$:
 see Figure 3.
\vskip 1mm

\begin{figure}
\includegraphics[scale=.65]{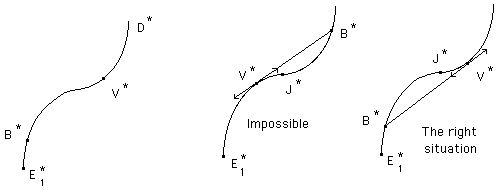}
\label{Figure 3}
\caption{\hfil \hskip 26mm \textbf{Fig. 3bis}}
\end{figure}
Now we will locate the two inflection points $J^*$ (on the positive part of the branch) and $I^*$ (on the negative part), which have the same abscissa and 
opposite second coordinates (see the proof of lemma 9). We know that the three arcs $\Arc{D^*I^*}$, $\Arc{I^*J^*}$ and $\Arc{J^*D^*}$ are strictly convex, 
and that we have the relation $B^*=V^**V^*$ which means that the tangent to the curve at $V^*$ cuts again the curve at the point $B^*$. If we suppose that $J^*$ was on the arc
$\Arc{V^*D^*}$, by the convexity of the arc $\Arc{E_1^*J^*}$, the tangent at $V^*$ would cut the curve at the point $B^*\in\Arc{J^*D^*}\subset\Arc{V^*J^*D^*}$, which 
is impossible, because we have $x_{B^*}<x_{V^*}$ by (41) (see Figure 3bis).

So the order of the points on the arc $\Arc{E_1^*D^*}$ is strict and is $E_1^*<B^*<J^*< V^*<D^*$, where $``P<Q"$ means that $P$ is at left of $Q$ and $P\not=Q$.
 For transporting the different connected arcs into the initial curve $\overline{{\cal C}_K}$, we use an easy fact : 
\vskip 1mm
\noindent\textit {Fact.}  The slope $p$ of the tangent to $\overline{{\cal C}_K}$ at the point $B=(0,-d,1)$ is, for $K>0$,
\begin{equation}
p=-(d^2+Kd+1)<-1.
\end{equation}
\vskip 1mm
Now, with this tangent, it is clear that the arc $\Arc{E_1BJV}$ is as in Figure 1, and also the arc $\Arc{E_1AIH}$ by symmetry.

For the two symmetric and convex arcs $\Arc{VD}$ and $\Arc{DH}$, we make a little calculation: we put, in the equation of ${\cal C}_K$, $x+y-K=t$; 
we obtain $ty^2+t(t-K)y-(t+K+d)=0$. So, if $t\to 0$, we have
\[y=±\frac{\sqrt{K+d}}{\sqrt {t}}(1+o(1)),\]
which is defined only for $t>0$, 
with $y\to±\infty$ when $t\to0_+$. So the two connected convex arcs with asymptotes $x+y-K=0$, $x=0$ or $y=0$ are above 
the line $x+y-K=0$: it is the case at $\pm\infty$, and this line does not cut the curve. 
So we have the form of Figure 1. \hskip 4mm \qed
\vskip 2mm
\begin{nota}\label{CourbeBizarre}
The necessity of Lemma \ref{FormCubic} is justified by the other forms of the curve in some cases (in fact when $K\leq-d$), see for example 
the following Figure 4.
\end{nota}
\begin{figure}
\includegraphics[scale=.65]{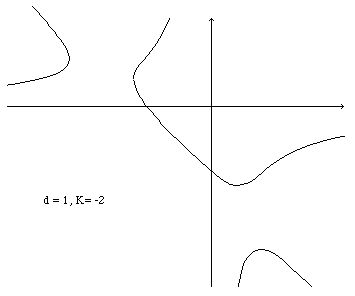}
\label{Figure 4}
\caption{}
\end{figure}

\begin{nota}\label{CoordHdoubletilda}
It results of the relation $\widetilde{B}=\widetilde{V}*\widetilde{V}$ (which comes by translation from the formula 
$B^*=V^**V^*$) that we have 
\[\ds\widetilde{\widetilde{H}}=\psi(\widetilde{H})=\widetilde{H}\build{+}_{\widetilde{V}}^{ }\omega=(\widetilde{H}*\omega)*\widetilde{V}=
\widetilde{V}*\widetilde{V}=\widetilde{B}=(\frac{K^3}{12}-\frac{2K}{3},dK^{3/2},1)\]
(see (41) and the proof of Proposition \ref{limitealinfini}).
\end{nota}

\vskip 2mm
\section{Chaotic behaviour of the dynamical system (2)}

In this part we will see that the dynamical system (2) in $\mathbb R_*^{+^2}$ associated to the homographic system of difference 
equations (6) has uniform sensitiveness 
to initial conditions on every compact set not containing the equilibrium $L$.
\vskip 2mm
\begin{teo}\label{Chaos}
For every compact set ${\cal K}\subset\mathbb R_*^{+^2}$ not containing the equilibrium $L$, it exists a number $\delta({\cal K})>0$ such that 
for every point $M\in{\cal K}$ and every neighborhood $W$ of $M$ it exists $M'\in W$ such that $ {dist}(F^n(M),F^n(M')\geq\delta({\cal K})$ 
for infinitely many integers $n$.
\end{teo}
\vskip 1mm
Of course, in this assertion, ``\textit {dist}" denote the euclidean distance in $\mathbb R_*^{+^2}$.
\vskip 2mm
The proof will use a general topological result (probably well known) on ``dynamical systems fibring in rotations on $\mathbb T$", the critical point for using 
this result
being the proof of some uniformity in the inverse conjugation of previous Section \ref{sec:3}.

\begin{propo}\label{FibreCercles}
Let $X$ be a metric space. Let be also 
$\theta:X\to\mathbb T=\mathbb R/\mathbb Z$ a continuous map such that for every non-empty open set $U$, the set $\theta(U)$ contains a non-empty open set. Define the map 
\[g:X\times\mathbb T\to X\times\mathbb T:(x,\a)\to(x,\a+\theta(x)).\]
Then the dynamical system $(X\times\mathbb T,g)$ has $\delta$-sensitiveness to initial conditions for every $\delta\in]0,1/2[$.
\end{propo}

\begin{figure}
\includegraphics[scale=.65]{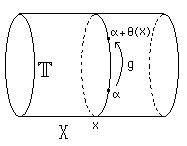}
\label{Figure 5}
\caption{}
\end{figure}

\begin{proof}
Although this is probably a classical result, we give a short proof for the completeness. Let be $M:=(x_0,\alpha_0)\in X$, and let be 
$W:=U\times I$ an open neighborhood of $(x_0,\alpha_0)$. Let be $J$ a non-empty open interval contained in $\theta(U)$. 
It is possible 
to find in $J$ a point $\theta_1$ such that in $\mathbb T=\mathbb R/\mathbb Z$ the number $\lvert\theta_1-\theta(x_0)\rvert$ is irrational and in ]0,1[. 
It exists $x'\in U$ such that
$\theta(x')=\theta_1$. We put $M'=(x',\a_0)\in W$, and calculate the distance between the points $g^n(M)$ and $g^n(M')$: 
$\widetilde{d}\big(g^n(M),g^n(M')\big)=d_X(x,x')+||n(\theta_1-\theta(x_0))||$ ($\widetilde{d}$ denote the $\ell^1$-distance in the product $X\times\mathbb T$, 
and $||.||$ the distance to the point 0 in $\mathbb R/\mathbb Z$). Let be $\delta\in]0,1/2[$. Since $\theta_1-\theta(x_0)$ is irrational, for infinitely many values of $n$ we 
have $||n(\theta_1-\theta(x_0))||>\delta$, 
and so $\widetilde{d}\big(g^n(M),g^n(M')\big)\geq\delta$ for infinitely many $n$: this is the $\delta$-sensitiveness to initial conditions of 
the dynamical system $(X\times\mathbb T,g)$. \hskip 4mm\qed
\end{proof}
\vskip 2mm
Now it is easy to prove the following fact:
\vskip 2mm
\noindent\textit {\textbf{Fact.}} Let be $(X,g)$ a dynamical system which has $\delta$-sensitiveness to initial conditions, where $X$ is a \textit {compact} metric space. Let be $Y$
another compact metric space, and $h:X\to Y$ a homeomorphism. Then the conjugated dynamical system $(Y,\widetilde{g})$ by $h$ 
($\widetilde{g}=h\circ g\circ h^{-1}$) has $\eta$-sensitiveness to initial conditions for some $\eta>0$.

\vskip 2mm
The last argument for the proof of theorem 4 will be the following result.
\vskip 2mm
\begin{lem}\label{ContWeierstrass}
Let be $\wp_{_K}$ the Weierstrass' function built with the numbers $g_2(K)$ and $g_3(K)$ of formulas (20) and (21), and defined by the formula
\begin{equation}
\wp_{_K}(z)=\frac{1}{ z^2}+\sum_{\l\in\Lambda_K,\,\l\not=(0,0)}\Big(\frac{1}{(z-\l)^2}-\frac{1}{\l^2}\Big),
\end{equation}
where $\Lambda_K$ is the lattice of the points of $\mathbb C$ defined by
\[\Lambda_K:=\{2p\omega_1(K)+2iq\omega_2(K)\,|\,(p,q)\in\mathbb Z^2\},\]
with
\begin{equation}
\omega_1(K)=\frac{1}{\sqrt{e_1-e_3}}\int_0^{+\infty}\frac{\textrm {d}u}{\sqrt{(1+u^2)(1+\varepsilon u^2)}}
\end{equation}
and
\begin{equation}
\omega_2(K)=\frac{1}{\sqrt{e_1-e_3}}\int_0^{\pi/2}\frac{\textrm {d}u}{\sqrt{1-\varepsilon\sin^2u}},
\end{equation}
where $e_1,\,e_2,\,e_3,\,\varepsilon$ are as in formula (24).
Then we have the property:  
\[ \wp_{_K}(2x\omega_1(K)+i\omega_2(K))\to\wp_{_{K_0}}(2x\omega_1(K_0)+i\omega_2(K_0))\,\,\,\textit {uniformly\,\, for\,\,}x\in[0,1]\]
when $K\to K_0$, and the same property for the derivative $\wp'_{_K}$.
\end{lem}
\vskip 2mm
\noindent\textit{\textbf{ Sketch of the proof of the lemma.}}  For formulas (44) and (45) we refer to [1] or [2]. The segments $H_K:=\{2x\omega_1(K)+i\omega_2(K)\,|\,x\in[0,1]\}$, when $K\to K_0$, 
are 
contained in a fixed compact set disjoint from $\Lambda_{K_0}$ and from the $\Lambda_K$ if $K$ is near to $K_0$. So, by some calculations, 
one can easily have uniformly for $x\in[0,1]$ an  upper bound 
for  the remainder of the series (43). For the finite part of the sum, we use the continuity in $K$, uniformly with respect to $x$, 
of each of its terms, by using 
the continuity of the functions $K\mapsto e_j(K)$, $K\mapsto\omega_i(K)$. \hskip 4mm \qed
\vskip 3mm
\noindent\textit{\textbf{ Sketch of the proof of the theorem.}} 
 It is well known (see [1]) that for $x$ varying in $[0,1]$ the formula
\[{\cal P}_K(x):=\big(\wp_{_K}(2x\omega_1(K)+i\omega_2(K),\wp'_{_K}(2x\omega_1(K)+i\omega_2(K)\big)\] 
gives a one-to-one parametrization of the bounded component $\Gamma_K^+$ of  $\Gamma_K$. From now on we denote $\phi_K$ and $\psi_K$ the maps 
defined by formulas (15), (16), (18) and (22), for marking their dependance on $K$. Now, the map $z\mapsto(\psi_K\circ \phi_K)^{-1}(z)=
(\psi_K\circ{\cal T}_3\circ{\cal T}_2\circ{\cal T}_1)^{-1}(z)$ depends continuously on $K$, uniformly for $z$ in 
a given compact in $\mathbb R_*^{+^2}$ (where the $\Gamma_K^+$ are situated for $K$ near $K_0$). 

So, if we choose $K_1,\,K_2$ such that $K_m<K_1<K_2$, the map $h$ defined on $[K_1,K_2]\times\mathbb T$ by
\[h(K,x)=(\psi_K\circ \phi_K)^{-1}\circ{\cal P}_K(x)\]
is a homeomorphism of $[K_1,K_2]\times\mathbb T$ onto the set 
\[\Delta(K_1,K_2):=\{K_1\leq G\leq K_2\}=\bigcup_{K_1\leq K\leq K_2}{\cal C}_K^+.\]

\begin{figure}
\includegraphics[scale=.65]{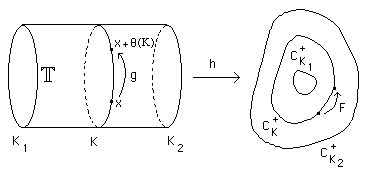}
\label{Figure 6}
\caption{}
\end{figure}

But the map $F:\Delta(K_1,K_2)\to\Delta(K_1,K_2)$ is conjugated by $h$ to the map $g:[K_1,K_2]\times\mathbb T\to[K_1,K_2]\times\mathbb T:(K,x)\mapsto(K,x+\theta_d(K))$ 
(see Theorem \ref{ConjRot} and figure 6). So by Lemma \ref{ContWeierstrass}, 
the fact and Proposition \ref{FibreCercles}, the dynamical system $\big(\Delta(K_1,K_2),F\big)$ has a $\delta$-sensitiveness to initial conditions for some $\delta>0$.
Now, if a compact set ${\cal K}\subset\mathbb R_*^{+^2}$ does not contain the equilibrium $L$, put $\ds K_1=\min_{\cal K}G$ and $\ds K_2=\max_{\cal K}G$. Then
${\cal K}\subset\Delta(K_1,K_2)$, and $F_{\,| _{\cal K}}$ has a $\delta$-sensitiveness to initial conditions. \hskip 4mm\qed
\vskip 1mm
\begin{nota}\label{Pointwise}
The proof of theorem 4 gives in fact an improvement to the assertion on ``pointwise" chaotic behavior of the dynamical systems 
studied in [2], [3], [4].
\end{nota}

\vskip 6mm
\noindent\textbf{Appendix.\hskip 4mm An assisted computer proof that the map $\psi$ defined by formula (22) is an isomorphism between the two chord-tangent group laws }
\vskip 2mm
We start with a standard regular cubic ${\cal C}$ in Weierstrass' form: $y^2=4x^3+ax+b$, with a given point $Z$ on it, we denote $\omega$ the infinite point of ${\cal C}$ in
vertical direction, and suppose $Z\not=\omega$. We denote $\build{+}_{Z}^{ }$ the addition for the chord-tangent law on ${\cal C}$ with zero element $Z$. So we have the following result.

\begin{propo}\label{Chord-tangent}
The map $\psi:{\cal C}\to{\cal C}:M\mapsto M\build{+}_{Z}^{ }\omega$ is an isomorphism of the chord-tangent law on ${\cal C}$ with 
unit element $Z$ on the standard 
chord-tangent law on ${\cal C}$ with unit element $\omega$.
\end{propo}

\vskip 1mm
The following result is asserted without proof in [10] page 21.
\vskip 1mm

\begin{corol}\label{corolChord-tangent}
Let be $A$ and $Z$ two different points on the regular cubic $y^2=4x^3+ax+b$. Then the two chord-tangent group laws with unit elements
 $Z$ and $A$ are isomorphic.
\end{corol}

\vskip 1mm
\noindent\textit{\textbf{A computer assisted proof of  Proposition \ref{Chord-tangent}}}
First, we remark that the map $\psi$ is obviously an isomorphism of the group $({\cal C},\build{+}_{Z}^{ },Z)$ onto
 the group
$({\cal C},\times,\omega)$, where 
\begin{equation}
P \times Q:= P\build{+}_{Z}^{ }Q\build{+}_{Z}^{ }\big(\build{-}_{Z}^{ }\omega\big).
\end{equation}
 We remark also that we have
 \begin{equation} 
\big(\build{-}_{Z}^{ }\omega\big)=(Z*Z)*\omega:=W.
 \end{equation}
Now we have to show that the law $\times$ coincides with the law $\build{+}_{\omega}^{ }$, that is 
\[(P*Q)*\omega=\big\{[(P*Q)*Z]*W\big\}*Z\]
for every $P,\,Q$. This is equivalent to the relation
 \begin{equation}
\forall R\in{\cal C},\,\,\,R*\omega=[(R*Z)*W]*Z.
 \end{equation}
 For proving relation (48) we use Maple, and present a sequence of calculation instructions which gives the result, with comments. 
\vskip 1mm
We put $Z=(u,v)$, with $v\geq 0$ (this is possible by the symmetry of the curve). We put $W=(U,V)$, and denote $R=(x,y)$. We have $v=\sqrt{4u^3+au+b}$, and
 two possibilities
$y=\sqrt{4x^3+ax+b}$ or $y=-\sqrt{4x^3+ax+b}$, so we write two lists of instructions, depending on the sign of $y$. We denote 
\[R*Z:=R_1=(x_1,y_1),\,\,\,R_1*W:=R_2=(x_2,y_2),\,\,\,R_2*Z:=R_3=(X,Y).\]
Our goal is to prove that $R_3=R*\omega$, that is $X=x$ and $Y=-y$.
We remark that if a line $y=px+q$ cuts the curve at three points, their abscisses satisfy the relation $x_1+x_2+x_3=p^2/4$ (see the proof of Proposition \ref{limitealinfini}).

\vskip 1mm
The following figure shows the geometric construction.

\begin{figure}
\includegraphics[scale=.65]{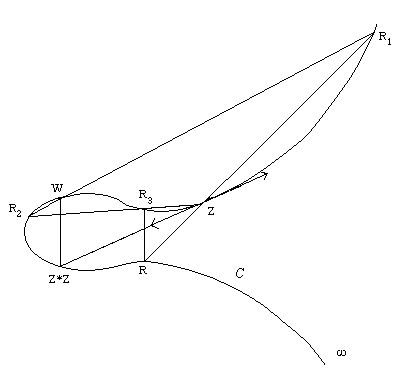}
\label{Figure 7}
\caption{}
\end{figure}

\noindent\textit {First list of Maple instructions} 

\vskip 3mm
\begin{align*}
\noindent v:&=\textrm {sqrt}(4*u^3+a*u+b);
\hskip 26mm[Z=(u,v)\in{\cal C}]\hskip 26mm
\\
\noindent y:&=\textrm {sqrt}(4*x^3+a*x+b);
\hskip 26mm[R=(x,y)\in{\cal C}]\hskip 26mm
\end{align*}

\vskip -5mm
\begin{align*}
\left.
\begin{aligned}
p:&=\textrm {simplify}((12*u^2+a)/(2v));
 \\
q:&=\textrm {simplify}(v-p*u);
\end{aligned}
\right\}
\hskip 14mm[y=px+q\hbox{ is the tangent to ${\cal C}$ at }Z]
\end{align*}

\vskip -5mm
\begin{align*}
\left.
\begin{aligned}
U:&=\textrm {simplify}(p^2/4-2*u);
\\
V:&=\textrm {simplify}(-(p*U+q)); 
\end{aligned}
\right\}
\hskip 21mm[\hbox{coordinates of } W]\hskip 20mm
\end{align*}

\vskip -5mm
\begin{align*}
\left.
\begin{aligned}
p[1]:&=\textrm {simplify}((y-v)/(x-u));
\\
q[1]:&=\textrm {simplify}(v-p[1]*u);
\end{aligned}
\right\}
\hskip 15mm[y=p_1x+q_1\hbox{ is the line }(RZ)]
\end{align*}
\vskip -5mm
\begin{align*}
\left.
\begin{aligned}
x[1]:&=\textrm {simplify}(p[1]^2/4-x-u);
\\
y[1]:&=\textrm {simplify}(p[1]*x[1]+q[1]);
\end{aligned}
\right\}
\hskip 15mm[\hbox{coordinates of }R_1]\hskip 16mm
\end{align*}

\vskip -5mm

\begin{align*}
\left.
\begin{aligned}
p[2]:&=\textrm {simplify}((y[1]-V)/(x[1]-U));
\\
q[2]:&=\textrm {simplify}(V-p[2]*U);
\end{aligned}
\right\}
\hskip 8mm[y=p_2x+q_2\hbox{ is the line }(R_1W)]
\end{align*}

\vskip -5mm
\begin{align*}
\left.
\begin{aligned}
x[2]:&=\textrm {simplify}(p[2]^2/4-x[1]-U);
\\
y[2]:&=\textrm {simplify}(p[2]*x[2]+q[2]);
\end{aligned}
\right\}
\hskip 17mm[\hbox{coordinates of }R_2]\hskip 28mm
\end{align*}

\vskip -5mm
\begin{align*}
\left.
\begin{aligned}
p[3]:&=\textrm {simplify}((y[2]-v)/(x[2]-u));
\\
q[3]:&=\textrm {simplify}(v-p[3]*u);
\end{aligned}
\right\}
\hskip 10mm[y=p_3x+q_3\hbox{ is the line }(R_2Z)]
\end{align*}

\vskip -5mm
\begin{align*}
\left.
\begin{aligned}
X:&=\textrm {simplify}(p[3]^2/4-u-x[2]);
\\
Y:&=\textrm {simplify}(p[3]*X+q[3]);
\end{aligned}
\right\}
\hskip 20mm[\hbox{coordinates of }R_3]\hskip 14mm
\end{align*}

\vskip 2mm
At this point, we see that $X$ has the form $\ds X=\frac{N.x}{ D^2}$. So we continue with

\vskip 1mm
H:= simplify($D^2$);

A:= simplify(X$*$H-H$*$x);

\vskip 1mm
and obtain $A=0$, that is $X=x$. Now we substitute $x$ to $X$ in the formula giving $Y$, obtain $Y'$, and then make
\vskip 2mm
B:= simplify($Y'$$*$sqrt(4$*$$x^3$+a$*$x+b)+4$*$$x^3$+a$*$x+b);
\vskip 2mm
We obtain $B=0$, that is $Y=Y'=-\sqrt{4x^3+ax+b}=-y$, and so $X=x$ and $Y=-y$, that is the result.
\vskip 2mm
Now we write the same list of instructions, but with $y=-\sqrt{4x^3+ax+b}$ and an obvious modification in $B$, and obtain also the same result.  \hskip 4mm \qed

\vskip 6mm
\noindent{\bf References}

\vskip 2mm
\noindent [1] P. Appel et E. Lacour, \textit{ Principes de la th\'eorie des fonctions elliptiques}, 1897, Paris, Gauthier-Villars.
\vskip 1mm
\noindent [2] G. Bastien and M. Rogalski, \textit{Global Behaviour of the Solutions  of Lyness' Difference Equations}, 
J. of Difference Equations and Appl., 2004, vol. 10, p. 977-1003. 
\vskip 1mm
\noindent[3] G. Bastien and M. Rogalski, \textit{ On some algebraic difference equations $u_{n+2}u_n=\psi(u_{n+1})$ in $\mathbb R^+_*$,
 related to families of conics or cubics : generalization of the Lyness' sequences}, Journal of Mathematical Analysis and Applications 300 (2004), 303-333.
\vskip 1mm
\noindent[4] G. Bastien and M. Rogalski, \textit {On the algebraic difference equations $u_{n+2}+u_n=\psi(u_{n+1})$ in $\mathbb R$, related to a family of elliptic 
quartics in the plane}, J. Math.Anal. Appl. 326 (2007) 822-844.
\vskip 1mm
\noindent[5] G. Bastien and M. Rogalski,  \textit{ Level sets lemmas and unicity of critical point of invariants, tools for local stability and topological properties of dynamical systems}, Sarajevo Journal of Math., vol. 8 (21) (2012), 273-282.
\vskip 1mm
\noindent[6] J. Duistermaat,  \textit{Discrete Integrable Systems. QRT Maps and Elliptic Surfaces}, Springer (2010).
\vskip 1mm
\noindent[7] D. Husemoller, \textit{Elliptic Curves}, 1987, Springer-Verlag, USA New-York.
\vskip 1mm
\noindent[8] G.R.W. Quispel, J. A. G. Roberts  and C. J. Thompson, \textit{ Integrable mappings and soliton equations}, Physics Letters A:
126 (1988), 419-421. 
\vskip 1mm
\noindent[9] G. Robin,  \textit{Estimation de la fonction de Tchebychef $\Theta$ sur le k-i\`{e}me nombre premier et
grandes valeurs de la fonction $\omega(n)$ nombre de diviseurs premiers de $n$},  Acta Arithmetica XLII (1983), p. 367-389.
\vskip 1mm
\noindent[10] J. H. Silverman and J. Tate, \textit{Rational points on elliptic curves}, Springer.
\vskip 1mm
\noindent[11] G. Tenenbaum, \textit{Introduction \`a la th\'eorie analytique et probabiliste des
nombres}, Cours sp\'ecialis\'e n$^{\rm o}$1, Collection Soci\'et\'e Math\'ematique de France, 1995, Paris. 
\vskip 1mm
\noindent[12] E. C. Zeeman,  \textit{ Geometric unfolding of a difference equation}, unpublished paper (1996).

\end{document}